\documentclass[11pt,a4paper]{amsart}
\usepackage{tikz}

\newtheorem{theorem}{Theorem}[section]
\newtheorem{lemma}[theorem]{Lemma}
\newtheorem{corollary}[theorem]{Corollary}

\newcommand{\RR}{\mathbb R}
\newcommand{\eps}{\varepsilon}

\begin{document}\suppressfloats
\title{Generalised Mycielski graphs and the Borsuk--Ulam theorem}
\author{Tobias M\"uller}
\thanks{The first author was partially supported by NWO grants 639.032.529
and 612.001.409. Part of this work was carried out while this  author visited
Laboratoire G-SCOP supported by ANR Project Stint (ANR-13-BS02-0007) and
LabEx PERSYVAL-Lab (ANR-11-LABX-0025-01).}
\address{Johann Bernoulli Institute, Groningen University, The Netherlands}
\email{tobias.muller@rug.nl}
\author{Mat\v ej Stehl\'ik}
\thanks{The second author was partially supported by ANR project Stint
(ANR-13-BS02-0007), ANR project GATO (ANR-16-CE40-0009-01), and by LabEx
PERSYVAL-Lab (ANR-11-LABX-0025). Part of this work was carried out while this
author visited Utrecht University, supported by NWO grant 639.032.529.}
\address{Laboratoire G-SCOP, Univ.\ Grenoble Alpes, France}
\email{matej.stehlik@grenoble-inp.fr}
%\date{\today}

\begin{abstract}
  Stiebitz determined the chromatic number of generalised
  Mycielski graphs using the topological method of Lov\'asz,
  which invokes the Borsuk--Ulam theorem. Van Ngoc and Tuza
  used elementary combinatorial arguments to prove Stiebitz's
  theorem for $4$-chromatic generalised Mycielski graphs, and
  asked if there is also an elementary combinatorial proof for
  higher chromatic number. We answer their question by showing
  that Stiebitz's theorem can be deduced from a version of Fan's
  combinatorial lemma. Our proof uses topological terminology,
  but is otherwise completely discrete and could be rewritten to
  avoid topology altogether. However, doing so would be somewhat
  artificial, because we also show that Stiebitz's theorem is
  equivalent to the Borsuk--Ulam theorem.
\end{abstract}
\maketitle

\section{Introduction}

The Mycielski construction~\cite{Myc55} is one of the earliest and
arguably simplest constructions of triangle-free graphs of
arbitrary chromatic number. Given a graph $G=(V,E)$, we let
$M_2(G)$ be the graph with vertex set $V \times \{0,1\} \cup \{z\}$,
where there is an edge $\{(u,0),(v,0)\}$ and $\{(u,0),(v,1)\}$ whenever
$\{u,v\} \in E$, and an edge $\{(u,1),z\}$ for all $u \in V$.
It is an easy exercise to show that the chromatic number increases
with each iteration of $M_2(\cdot)$.

The construction was generalised by Stiebitz~\cite{Sti85} (see
also~\cite{SS89,GJS04}), and independently by Van Ngoc~\cite{Van87} (see also~\cite{VNT95}),
in the following way. Given a graph $G=(V,E)$
and an integer $r \geq 1$, we define $M_r(G)$
as the graph with vertex set $V \times \{0,\ldots,r-1\} \cup \{z\}$,
where there is an edge $\{(u,0),(v,0)\}$ and $\{(u,i),(v,i+1)\}$ whenever
$\{u,v\} \in E$, and an edge $\{(u,r-1),z\}$ for all $u \in V$.
The construction is illustrated in Figure~\ref{fig:gen-mycielski}.

\begin{figure}
\begin{center}
\begin{tikzpicture}[scale=0.8,
                    vertex/.style={circle, draw=black, fill, inner sep=0mm, minimum size=4pt},
                    edge/.style={semithick}]
\foreach \i in {0,...,4}
{
  \draw (-1,\i-2) coordinate (a\i) {}
        (0,\i-2) coordinate (b\i) {}
        (1,\i-2) coordinate (c\i) {};
}
\draw (2,0) coordinate (z){};

\foreach \i/\j in {0/1,1/2,2/3,3/4,4/0,1/0,2/1,3/2,4/3,0/4}{
  \draw[edge] (a\i)--(b\j);
  \draw[edge] (b\i)--(c\j);
}
\draw[edge] (a0)--(a1)--(a2)--(a3)--(a4)
            (a0)..controls (-2,-1) and (-2,1)..(a4);

\foreach \i in {0,...,4}
{
  \draw[edge] (z)--(c\i) {};
}

\foreach\i in {0,...,4}
{
  \draw (a\i) node[vertex] {};
  \draw (b\i) node[vertex] {};
  \draw (c\i) node[vertex] {};
}
\draw (z) node[vertex]{};

\end{tikzpicture}
\end{center}
\caption{The graph $M_3(C_5) \cong M_3(M_2(K_2))\in \mathcal M_4$.}\label{fig:gen-mycielski}
\end{figure}
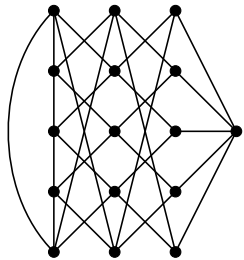

If $r>2$, it is no longer true that the chromatic number increases
with each iteration of $M_r(\cdot)$. For instance,
it can be shown that if $\overline{C_7}$ is the complement of the
$7$-cycle, then $\chi(M_3(\overline{C_7}))=\chi(\overline{C_7})=4$.
However, Stiebitz~\cite{Sti85} was able to show that the chromatic
does increase with each iteration of $M_r(\cdot)$ if we start with
an odd cycle, or some other suitably chosen graph.
For every integer $k \geq 2$, let us denote by $\mathcal M_k$ the
set of all `generalised Mycielski graphs' obtained from $K_2$ by
$k-2$ iterations of $M_r(\cdot)$, where the value of $r$ can
vary from iteration to iteration. Stiebitz~\cite{Sti85} (see
also~\cite{GJS04,Mat03}) proved the following.

\begin{theorem}[Stiebitz~\cite{Sti85}]
\label{thm:stiebitz}
  If $G \in \mathcal M_k$, then $\chi(G) \geq k$.
\end{theorem}

Stiebitz's proof is based on Lov\'asz's~\cite{Lov78} bound on the chromatic
number in terms of the connectedness of the neighbourhood complex, which
Lov\'asz developed to prove Kneser's conjecture (see~\cite{Mat03} for a
comprehensive account). Lov\'asz's bound uses the following result of
Borsuk~\cite{Bor33}, usually known in the literature as the
\emph{Borsuk--Ulam theorem}.
\begin{theorem}[Borsuk~\cite{Bor33}]
\label{thm:borsuk-ulam}
  There exists no continuous antipodal mapping $f:S^n \to S^{n-1}$; that is,
  a continuous mapping such that $f(-x)=-f(x)$ for all $x \in S^n$.
\end{theorem}

To this day, no combinatorial proof of Theorem~\ref{thm:stiebitz} is known
(see~\cite[pp.~133]{Mat03}), except for the case $k=4$~\cite{VNT95}. At the
end of their paper, Van Ngoc and Tuza~\cite{VNT95} propose the following problem:
\begin{quote}\small
  Finally, we would like to invite attention to the problem that no elementary
  combinatorial proof is known so far for the general form of Stiebitz's theorem,
  yielding graphs of arbitrarily large chromatic number and fairly large odd girth.
\end{quote}

The answer to the problem depends on the interpretation of `elementary combinatorial
proof'. Does it mean a proof that is `discrete' and does not rely on continuity?
Or does it mean a `graph theoretic' proof which avoids any topological concepts,
such as triangulations of spheres? 

In this note we will give a new discrete proof of Theorem~\ref{thm:stiebitz} based
on a generalisation, due to Prescott and Su~\cite{PS05}, of a classical lemma of
Fan~\cite{Fan52}, and on a result of Kaiser and Stehl\'ik~\cite{KS15}. Since the
proofs of both these theorems are discrete, this provides a discrete proof of
Theorem~\ref{thm:stiebitz}.

Triangulations of spheres are central to our proof, and rewriting the proof
so as to avoid any topological concepts (as Matou\v sek~\cite{Mat04} has done for the
Lov\'asz--Kneser theorem) is certainly possible, but seems somewhat artificial.
Indeed, we show that Theorem~\ref{thm:borsuk-ulam} follows fairly easily from
Theorem~\ref{thm:stiebitz}.

We would like to point out that our proof of Theorem~\ref{thm:stiebitz}
leads to a new proof of Schrijver's~\cite{Sch78} sharpening of the Lov\'asz--Kneser
theorem~\cite{Lov78}, via the following result of Kaiser and Stehl\'ik~\cite{KS15}
(whose proof is entirely combinatorial). For a definition of $SG(n,k)$, we refer
the reader to~\cite{KS15} or~\cite{Mat03}.

\begin{theorem}[Kaiser and Stehl\'ik~\cite{KS15}]
  For all integers $k \geq 1$ and $n > 2k$, there exists a graph $G \in \mathcal M_k$
  homomorphic to $SG(n,k)$.
\end{theorem}

\section{Preliminaries}

Our graph theoretic terminology is standard and follows~\cite{BM08}.
For an excellent introduction to topological methods in combinatorics, and all the
topological terms used in this paper, see~\cite{Mat03}.

Prescott and Su~\cite{PS05} introduced flags of hemispheres to prove a slight generalisation
of Fan's combinatorial lemma~\cite{Fan52}. 
A \emph{flag of hemispheres} in $S^n$ is a sequence $H_0 \subset \cdots \subset H_n$
where each $H_d$ is homeomorphic to a $d$-ball, $\{H_0,-H_0\}$ are antipodal points,
$H_n \cup -H_n = S^n$, and for $1 \leq d \leq n$,
\[
  \partial H_d = \partial(-H_d) = H_d \cap -H_d = H_{d-1}\cup -H_{d-1} \cong S^{d-1}.
\]

The \emph{polyhedron} $|K|$ of a simplicial complex $K$ is defined as the union of all
of its simplices. We say that $K$ is a \emph{triangulation} of $|K|$ (or any space
homeomorphic to it).
A triangulation $K$ of $S^n$ is (centrally or antipodally) \emph{symmetric} if $\sigma \in K$ whenever
$-\sigma \in K$. A symmetric triangulation $K$ of $S^n$ is said to be \emph{aligned with
hemispheres} if we can find a flag of hemispheres such that for every $d$,
there is a subcomplex of the $d$-skeleton of $K$ that triangulates $H_d$.

Given a simplicial complex $K$ and a labelling (map)
$\lambda : V(K) \to {\mathbb Z}\setminus \{0\}$, we say a 
$d$-simplex $\sigma \in K$ is \emph{positive alternating} if it has labels
$\{+j_0,-j_1,+j_2, \ldots, (-1)^dj_d\}$, where $0<j_0<j_1< \cdots <j_d$.
The following version of Fan's lemma~\cite{Fan52} is a key ingredient of our proof.

\begin{theorem}[Prescott and Su~\cite{PS05}]
\label{thm:fan}
  Let $K$ be a symmetric triangulation of $S^n$ aligned
  with hemispheres, and let $\lambda:V(K)\to \{\pm1, \ldots, \pm k\}$
  be a labelling such that $\lambda(-v)=-\lambda(v)$ for every vertex
  $v \in V(K)$, and $\lambda(u)+\lambda(v) \neq 0$ for every edge $\{u,v\} \in K$. 
  Then there exists an odd number of
  positive alternating $n$-simplices. In particular, $k \geq n+1$.
\end{theorem}
We remark that the proof
in~\cite{PS05} is constructive and discrete, and that Fan's original
result~\cite{Fan52} imposes a more restrictive condition on the triangulation.

Suppose $K$ is a symmetric triangulation of $S^n$. A $2$-colouring
of $K$ is an assignment of two colours (black and white) to the vertices
of $K$. The $2$-colouring is said to be \emph{antisymmetric} if antipodal
vertices receive distinct colours, and it is \emph{proper} if no
$n$-simplex is monochromatic.

Given a symmetric triangulation $K$ of $S^n$ and a proper antisymmetric
$2$-colouring $\kappa$ of $K$, we denote by $\tilde G(K,\kappa)$ the graph
obtained from the $1$-skeleton $K^{(1)}$ by deleting all monochromatic
edges. If $\nu$ denotes the antipodal action on $\tilde G(K,\kappa)$,
we set $G(K,\kappa)=\tilde G(K,\kappa)/\nu$, and let
$p:\tilde G(K,\kappa) \to G(K,\kappa)$ be the corresponding projection.
Note that the graph $\tilde G(K,\kappa)$ is a bipartite double cover of
$G(K,\kappa)$.

The following theorem is an immediate consequence of
\cite[Lemma~3.2 and Theorem~6.1]{KS15}, where the results are stated
in terms of so-called quadrangulations of projective spaces.
\begin{theorem}[Kaiser and Stehl\'ik~\cite{KS15}]
\label{thm:mycielski-projective}
  Given $n\geq 1$, let $K$ be a symmetric triangulation of $S^n$ aligned with hemispheres, with a proper
  antisymmetric $2$-colouring $\kappa$.
  For any $r \geq 1$, there exists a symmetric triangulation $K'$ of $S^{n+1}$
  aligned with hemispheres,
  with a proper antisymmetric $2$-colouring $\kappa'$ such that
  $G(K',\kappa')\cong M_r(G(K,\kappa))$.
\end{theorem}

\section{A combinatorial proof of Theorem~\ref{thm:stiebitz}}

Our proof of Theorem~\ref{thm:stiebitz} is based on the following
corollary of Theorem~\ref{thm:fan}.

\begin{corollary}
\label{cor:fan}
  Let $K$ be a symmetric triangulation of $S^n$ aligned
  with hemispheres, and let $\lambda:V(K)\to \{\pm1, \ldots, \pm (n+1)\}$
  be a labelling such that $\lambda(-v)=-\lambda(v)$ for every vertex
  $v \in V(K)$, and every $n$-simplex has vertices of both signs. Then
  there exists an edge $\{u,v\} \in K$ such that $\lambda(u)+\lambda(v) = 0$.
\end{corollary}

\begin{proof}
  Let $K, \lambda$ be as in the corollary and suppose, for the sake of contradiction, that
  $\lambda(u)+\lambda(v) \neq 0$ for every edge $\{u,v\} \in K$. We now define a new labelling
  $\mu:V(K) \to \{\pm 1, \ldots, \pm (n+1)\}$ by $\mu(v)=(-1)^{|\lambda(v)|}\lambda(v)$.
  Observe that
  \[
    \mu(-v)=(-1)^{|\lambda(-v)|}\lambda(-v)=-(-1)^{|\lambda(v)|}\lambda(v)=-\mu(v),
  \]
  and if $\mu(u)=-\mu(v)$, then $\lambda(u)=-\lambda(v)$, and therefore
  $\mu(u)+\mu(v) \neq 0$ for every edge $\{u,v\} \in K$.
  Hence $\mu$ satisfies the hypothesis of Theorem~\ref{thm:fan}.
  Therefore, there is an odd number of positive alternating $n$-simplices,
  i.e., simplices labelled $\{1,-2,\ldots,(-1)^n n,(-1)^{n+1}(n+1)\}$ by $\mu$.
  Hence, there is an odd number of simplices labelled $\{1,2,\ldots,n+1\}$
  by $\lambda$.
  This contradicts the assumption that every $n$-simplex in $K$ has vertices of both
  signs. Hence, there exists an edge $\{u,v\} \in K$ such that $\lambda(u)+\lambda(v)=0$.
\end{proof}

We are now ready to prove Theorem~\ref{thm:stiebitz}.

\begin{proof}[Proof of Theorem~\ref{thm:stiebitz}]
  The case $k=2$ ($G = K_2$) and $k=3$ ($G$ is an odd cylce) are trivial,
  so assume $k > 3$ and let $G \in \mathcal M_k$.
  The graph $G$ is obtained from an odd cycle by $k-3$ iterations of $M_r(\cdot)$,
  where the value of $r$ can vary from iteration to iteration.
  By repeated applications of Theorem~\ref{thm:mycielski-projective}
  ($k-3$ applications to be exact), there exists a symmetric triangulation $K$ of
  $S^{k-2}$ aligned with hemispheres, and a proper antisymmetric $2$-colouring
  $\kappa$ such that $G \cong G(K, \kappa)$.
  (To see this, observe that $M_r(K_2)$ is isomorphic to the odd cycle $C_{2r+1}$, which is
  isomorphic to $G(K,\kappa)$, where $K$ is a symmetric triangulation of $S^1$---%
  i.e., a graph---isomorphic to the cycle $C_{4r+2}$, and $\kappa$ is a proper
  $2$-colouring of $K$. By choosing any pair of antipodal vertices of $K$ to
  be the hemispheres $H_0$ and $-H_0$, it is clear that $K$ is aligned with
  hemispheres.)
  Let us say the colours used in $\kappa$ are black and white.
  
  Consider any (not necessarily proper) $(k-1)$-colouring $c:V(G) \to \{1,\ldots,k-1\}$.
  By setting
  \[
    \lambda(v)=
    \begin{cases}
      +c(p(v)) &\text{ if $v$ is black}\\
      -c(p(v)) &\text{ if $v$ is white},
    \end{cases}
  \]
  we obtain an antisymmetric labelling $\lambda:V(K) \to \{\pm 1, \ldots, \pm(k-1)\}$
  such that every $(k-2)$-simplex has vertices of both signs.
  By Corollary~\ref{cor:fan}, there exists an edge $\{u,v\} \in K$ such that
  $\lambda(u)+\lambda(v) = 0$. Hence, the edge $\{p(u),p(v)\} \in E(G)$ satisfies
  $c(p(u))=|\lambda(u)|=|\lambda(v)|=c(p(v))$, i.e., $c$ is not a proper colouring of $G$.
  This shows that $\chi(G) \geq k$.
\end{proof}

\section{Equivalence of the theorems of Borsuk--Ulam and Stiebitz}

Let us recall the following construction due to Erd\H os and Hajnal~\cite{EH67}.
The \emph{Borsuk graph} $BG(n,\alpha)$ is defined as the (infinite) graph
whose vertices are the points of $\mathbb R^{n+1}$ on $S^n$, and the edges
connect points at Euclidean distance at least $\alpha$, where $0<\alpha<2$.
Using Theorem~\ref{thm:borsuk-ulam}, it can be shown that $\chi(G) \geq n+2$ (in fact the two
statements are equivalent, as noted by Lov\'asz~\cite{Lov78}). Furthermore, by using the
standard $(n+2)$-colouring of $S^n$ based on the central projection of a
regular $(n+1)$-simplex, it can be shown that $BG(n,\alpha)$ is $(n+2)$-chromatic
for all $\alpha$ sufficiently large. In particular, Simonyi and
Tardos~\cite{ST06} have shown that $BG(n,\alpha)$ is $(n+2)$-chromatic
for all $\alpha \geq \alpha_0$, where $\alpha_0=2 \sqrt{1-1/(n+3)}$.

We will need the following lemma.

\begin{lemma}
\label{lem:BG}
  For every $n \geq 0$ and every $\delta > 0$, there exists $G \in \mathcal M_{n+2}$
  and a mapping $f:V(G) \to S^n$ such that $\|f(u)+f(v)\| < \delta$,
  for every edge $\{u,v\} \in G$. In particular, $G \subset BG(n,\sqrt{4-\delta^2})$.
\end{lemma}

\begin{proof}
  The proof is by induction on $n$. We take $n=1$ as the base case, but we
  remark that the statement is also true for $n=0$, because $K_2$ is the only
  graph in $\mathcal M_2$, and the two vertices $u,v$ of $K_2$ can be placed at
  antipodal points of $S^0$, so $\|f(u)+f(v)\|=0$.
  
  To see that the statement is true for $n=1$, observe that $\mathcal M_3$
  is the family of odd cycles. The vertices of $C_{2r+1}$ can be
  mapped to $S^1$ so that $f(u)$ and $-f(v)$  are at angular distance
  $\pi/(2r+1)$, for every edge $\{u,v\} \in E(G)$. Therefore,
  $\|f(u)+f(v)\|=2\sin(\pi/(4r+2))$. As $r$ tends to infinity,
  $\|f(u)+f(v)\|$ tends to zero; in particular, for every $\delta>0$ there exists $r$
  such that $\|f(u)+f(v)\| < \delta$ for every $\{u,v\} \in E(C_{2r+1})$.
  
  Now suppose the theorem is true for $n \geq 1$. Fix a real number $\delta>0$.
  By the induction hypothesis, there exists $G \in \mathcal M_{n+2}$ and a
  mapping $f:V(G) \to S^n$ such that $\|f(u)+f(v)\| < \delta/2$ for every $\{u,v\} \in E(G)$.
  We let $r \geq 2$ be a large integer, to be specified shortly in the proof, and we
  define a mapping $\bar f:V(M_r(G)) \to S^{n+1}$ by setting:
  \begin{align*}
    \bar f(z)& :=(0, \ldots, 0,(-1)^r), \\
    \bar f((v,i)) & :=\left(f(v) \cos(\pi i/2r), (-1)^i\sin(\pi i/2r)\right).
  \end{align*}
  
  Fix an arbitrary edge $\{x,y\} \in E(M_r(G))$. We will show that $\|\bar f(x)+\bar f(y)\|<\delta$.
  First, if $x=(u,0)$ and $y=(v,0)$, for some $\{u,v\} \in E(G)$, then clearly
  $\bar f(x)=(f(u),0)$ and $\bar f(y)=(f(v),0)$, so $\|\bar f(x)+\bar f(y)\|=\|f(u)+f(v)\|<\delta/2$.
  
  Second, if $\{x,y\}=\{(u,i),(v,i+1)\}$, for some $\{u,v\} \in E(G)$, then
  applying the triangle inequality (twice) we get:
  \begin{align*}
    \|\bar f(x)+\bar f(y)\|
    &\leq \big\|f(u)\cos(\pi i/2r)+f(v)\cos(\pi (i+1)/2r)\big\|\\
    &\quad+|\sin(\pi i/2r)-\sin(\pi(i+1)/2r)| \\
    &\leq \|f(u)+f(v)\|\cdot|\cos(\pi i/2r)|\\
    &\quad+\|f(v)\|\cdot|\cos(\pi i/2r)-\cos(\pi (i+1)/2r)|\\
    &\quad+|\sin(\pi i/2r)-\sin(\pi(i+1)/2r)|\\
    & \leq \delta/2 + |\cos(t)-\cos(t + \eps)|+|\sin(t)-\sin(t + \eps)|,
  \end{align*}
  where $t = \pi i/2r, \eps = \pi/2r$. Since $\sin$ and $\cos$ are uniformly continuous,
  having chosen $r$ sufficiently large, we can assume that
  $|\cos(t)-\cos(t+\eps)|, |\sin(t)-\sin(t+\eps)| < \delta/4$ (for all
  $t\in \RR$ in fact). So $\|f(x)+f(y)\| < \delta$ as required.
  
  Finally, if $\{x,y\}=\{(u,r-1),z\}$, then we have 
  \begin{align*} 
  \|\bar f(x) + \bar f(y)\| &= \sqrt{\cos^2( \pi(r-1)/2r) + (1 - \sin(\pi(r-1)/2r)))^2} \\
  &< \delta,
  \end{align*}
  where the inequality holds provided $r$ was chosen sufficiently large, using that 
  $\cos( \pi(r-1)/2r )$ approaches $\cos(\pi/2) = 0$ and $\sin(\pi(r-1)/2r))$ approaches
  $\sin(\pi/2)=1$ as $r$ tends to infinity.
  
  Thus, we have now shown that, provided $r$ was chosen sufficiently large, 
  for every $\{x,y\} \in E(M_r(G))$ we have $\|\bar f(x)+\bar f(y)\|<\delta$.
  The lemma follows by induction.
\end{proof}

We will now show how Theorem~\ref{thm:borsuk-ulam} can be deduced from
Theorem~\ref{thm:stiebitz} and Lemma~\ref{lem:BG}.

\begin{proof}[Proof of Theorem~\ref{thm:borsuk-ulam}]
  Suppose there exists a continuous antipodal map $f:S^n\to S^{n-1}$.
  Set $\varepsilon=1/\sqrt{n+2}$.
  Since every continuous function on a compact set is uniformly continuous, there
  exists $\delta > 0$ such that if $\|x-y\|< \delta$, then
  $\|f(x)-f(y)\| < 2\varepsilon$.
  
  By Lemma~\ref{lem:BG}, there exists $G \in \mathcal M_{n+2}$ and
  a mapping $g:V(G) \to S^n$ such that $\|g(u)+g(v)\| < \delta$,
  for every edge $\{u,v\} \in E(G)$. Therefore, the mapping $f \circ g:V(G) \to S^{n-1}$
  satisfies $\|f(g(u))+f(g(v))\| < 2\varepsilon$, for every edge $\{u,v\} \in G$.
  Therefore, the Euclidean distance between $f(g(u))$ and $f(g(v))$ is
  \[
    \|f(g(u))-f(g(v))\|>2\sqrt{1-\varepsilon^2}=2\sqrt{1-1/(n+2)},
  \]
  so $G \subset BG(n-1,\alpha_0)$, and thus $\chi(G)\leq BG(n-1,\alpha_0)=n+1$.
  On the other hand, we have $\chi(G) \geq n+2$ by Theorem~\ref{thm:stiebitz}.
  This contradiction proves that there is no continuous antipodal map
  $f:S^n\to S^{n-1}$.
\end{proof}

\bibliographystyle{plain}
\bibliography{mycielski}
\end{document}